\documentclass[11pt, reqno]{amsart}
\pdfoutput=1

\usepackage{amsthm, amssymb, amsmath}
\usepackage[colorlinks=true, urlcolor=blue, citecolor=blue,  linkcolor=blue]{hyperref}
\usepackage{fullpage}
\usepackage{verbatim}
\usepackage{graphicx}
\usepackage{epsfig}
\usepackage{color}
\usepackage[all]{xy}
\newtheorem{thm}{Theorem}[section]
\newtheorem{lem}[thm]{Lemma}

\newtheorem{prop}[thm]{Proposition}

\newtheorem*{conjecture*}{Conjecture}

\newtheorem*{notThm*}{Not Yet a Theorem}
\newtheorem*{benTrick*}{Benedetto's Trick}
\newtheorem*{thm*}{Theorem}

\theoremstyle{remark} 
\newtheorem*{question*}{Question}

\newtheorem{remark}[thm]{Remark}

\theoremstyle{definition}

\numberwithin{equation}{section}  % number equations by section

\newcommand{\ZZ}{\mathbb{Z}}     % Integers
\newcommand{\RR}{\mathbb{R}}     % Reals
\newcommand{\Aff}{\mathbb{A}}      % Affine space
\newcommand{\QQ}{\mathbb{Q}}      %Rationals
\newcommand{\CC}{\mathbb{C}}      % Complex Numbers
\newcommand{\be}{\begin{equation}}
\newcommand{\ee}{\end{equation}}
\newcommand{\benn}{\begin{equation*}}
\newcommand{\eenn}{\end{equation*}}
\newcommand{\ba}{\begin{aligned}}
\newcommand{\ea}{\end{aligned}}
\newcommand{\bbm}{\begin{bmatrix}}
\newcommand{\ebm}{\end{bmatrix}}
\newcommand{\bpm}{\begin{pmatrix}}
\newcommand{\epm}{\end{pmatrix}}
\newcommand{\bi}{\begin{itemize}}
\newcommand{\ei}{\end{itemize}}
 
\newcommand{\an}[1]{\operatorname{an}}  % analytic space notation
\newcommand{\Hawaii}{Hawai\kern.05em`\kern.05em\relax i}

\newcommand{\PrePer}{\mathrm{PrePer}}

\title[Benedetto's Trick]{Benedetto's trick and existence of rational preperiodic structures for quadratic polynomials}

\author{Xander Faber}
\address{
Department of Mathematics \\
University of Hawaii \\
Honolulu, HI}
\email{xander@math.hawaii.edu}

\begin{document}
	\begin{abstract}
		We refine a result of R. Benedetto in $p$-adic analysis in order to exhibit infinitely many quadratic polynomials over $\QQ$ with a specified graph of rational preperiodic points. 
	\end{abstract}

\maketitle

\section{Introduction}
\label{intro}

	Let $K$ be a field. For $c \in K$, the quadratic polynomial $f_c(z) = z^2 + c$ is an endomorphism of the affine line $\Aff^1_K$. Via iteration, we may view it as a dynamical system on $\Aff^1(K) = K$. The set of $K$-rational preperiodic points --- those with finite forward orbit under $f_c$ --- will be denoted by $\PrePer(f_c, K)$. We equip this set with the structure of a directed graph by drawing an arrow from $x$ to $f_c(x)$ for each $x \in \PrePer(f_c,K)$. By Northcott's theorem, $\PrePer(f_c, K)$ is a \textit{finite} graph when $K$ is a number field. A well known conjecture of Morton and Silverman \cite{Morton_Silverman_1994} implies that there are only finitely many isomorphism types of directed graphs that can arise as one varies over all $c \in K$ and all number fields $K / \QQ$ of bounded degree. The problem of which isomorphism types can occur has been investigated extensively when $K = \QQ$ \cite{Walde_Russo,Morton_Arithmetic_Properties_II_1998,Poonen-Flynn-Schaefer_1997,Poonen_Preperiodic_Classification_1998,Stoll_6-cycles_2008}. Poonen has conjectured that there are exactly 12 isomorphism types that may arise over $\QQ$ \cite{Poonen_Preperiodic_Classification_1998}.  More recently, the case of quadratic extensions $K / \QQ$ has been studied by the author  and others \cite{JXD_2012,JXD_2013_I,Hutz_Ingram_2013}. 
	
	Poonen has remarked \cite[p.18]{Poonen_Preperiodic_Classification_1998} that the graphs shown in the figure below occur for infinitely many rational parameters $c$. This statement follows from Conjecture~2 and the results in \cite{Poonen_Preperiodic_Classification_1998}, but an unconditional proof is desirable. The goal of this note is to provide one.  

\begin{thm}
\label{Thm: Main}
	For each of the graphs $G$ shown 
below, there exist infinitely many parameters $c \in \QQ$ such that $\PrePer(f_c, \QQ) \cong G$. 
\begin{center}
	\begin{picture}(450,155)(-130,-25)
	
		% Graphs
		\put(-104,88){\includegraphics[scale=.8]{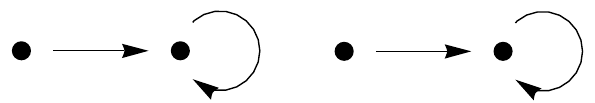}}
		\put(-97,45){\includegraphics[scale=.8]{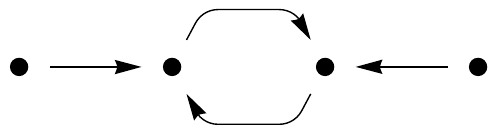}}
		\put(-101,-15){\includegraphics[scale=.7]{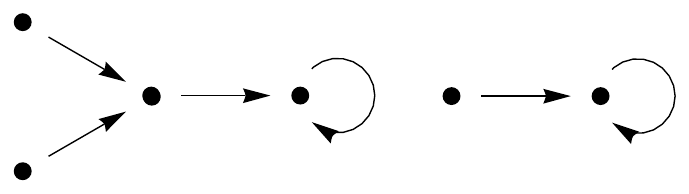}}
		\put(65,-5){\includegraphics[scale=.8]{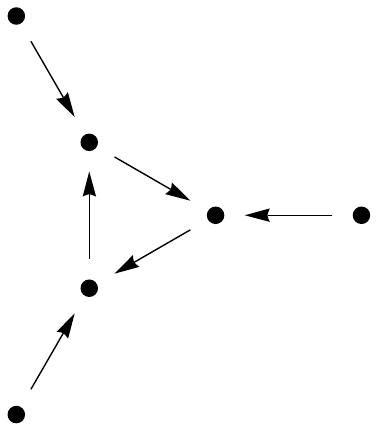}}
		\put(180,64){\includegraphics[scale=.75]{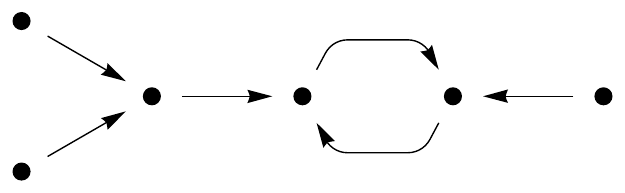}}
		\put(188,-15){\includegraphics[scale=.75]{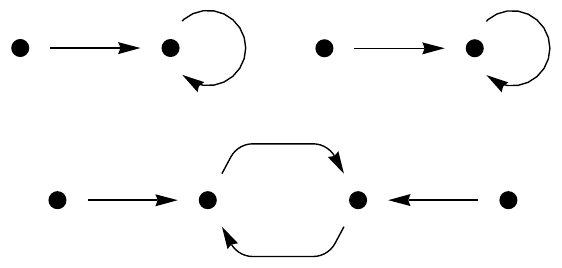}}
		
		% Boxes		
		\put(-130,122){$\underline{\hspace{450px}}$}
		\put(-130,-18){$\underline{\hspace{450px}}$}
		\put(-132,-20){\rotatebox{90}{$\underline{\hspace{140px}}$}}
		\put(42,-20){\rotatebox{90}{$\underline{\hspace{140px}}$}}
		\put(160,-20){\rotatebox{90}{$\underline{\hspace{140px}}$}}
		\put(320,-20){\rotatebox{90}{$\underline{\hspace{140px}}$}}								
		\put(-130,85){$\underline{\hspace{173px}}$}	
		\put(-130,40){$\underline{\hspace{173px}}$}	
		\put(162,60){$\underline{\hspace{159px}}$}
			
		% labels
		\put(-125, 111){{\scriptsize $\mathbf{4(1,1)}$}}
		\put(-125, 74){{\scriptsize $\mathbf{4(2)}$}}
		\put(-125, 29){{\scriptsize $\mathbf{6(1,1)}$}}
		\put(50, 105){{\scriptsize $\mathbf{6(3)}$}}		
		\put(167, 111){{\scriptsize $\mathbf{6(2)}$}}
		\put(167, 49){{\scriptsize $\mathbf{8(2,1,1)}$}}								
						
	\end{picture}
\end{center}
\end{thm}

\begin{remark}
	We follow the labeling convention in \cite{JXD_2012}. Write $N(\ell_1, \ell_2, \ldots)$ for a graph with $N$ vertices and directed cycles of lengths $\ell_1, \ell_2, \ldots$, in nonincreasing order. 
\end{remark}

%	We briefly explain a heuristic for the result in the case $G = 4(2)$ (to be concrete). If $c_0 \in \QQ$ is a rational parameter such that $\PrePer(f_{c_0}, \QQ)$ contains $G$, then there is a rational point $x \in \Aff^1(\QQ)$ such that $f_{c_0}^3(x) = f_{c_0}(x)$ and $f_{c_0}^2(x) \neq x$. Conversely, any rational pair $(c_0,x)$ that satisfies these two equations gives
	
%	$f^3_c(x) = f(x)$ system of equations of the form $f_c^n(z) = f_c^m(z)$ for various Given a finite directed graph $G$
	
We briefly explain a heuristic for the theorem. Let us say that a finite directed graph $G$ is \textbf{admissible} if there is a parameter $c \in \bar \QQ$ such that $\PrePer(f_c, \bar \QQ) \supset G$. For each admissible graph $G$, there exists a smooth connected projective algebraic curve $X_G$ over $\QQ$ such that all but finitely many of the algebraic points $x \in X_G(\bar \QQ)$ correspond to data $(x_1, \ldots, x_n, c) \in \bar \QQ^{n+1}$ for which $x_1, \ldots, x_n$ are preperiodic for $f_c$ and generate the graph $G$ under the action of $f_c$. In particular, if $(x_1, \ldots, x_n, c) \in X_G(\QQ)$, then $G \subset \PrePer(f_c,\QQ)$. For each graph $G$ as in the theorem, the curve $X_G$ is \textit{rational}, and hence admits infinitely many parameters $c \in \QQ$ for which $\PrePer(f_c, \QQ)$ \textit{contains} the graph $G$. 
% For each admissible graph $H$ that properly contains $G$, the corresponding curve $X_H$ naturally covers $X_G$ with degree at least~2, so that the set of points of $X_G(\QQ)$ arising from $X_H(\QQ)$ is thin (in the sense of Serre). 
While this is strong evidence for the theorem, the obstruction to making it a rigorous proof is that there are infinitely many admissible graphs $H$ properly containing $G$, and so it is possible that all but finitely many points of $X_G(\QQ)$ actually correspond to a larger graph~$H$. 
	
	Rather than taking an approach via rational points on curves, we start by constructing a sequence of parameters $(c_n)$ for which there exist preperiodic points in the desired configuration, and for which we can control the prime factors occurring in the denominator. (This step uses the Prime Number Theorem in one case and van der Corput's theorem on 3-term arithmetic progressions of primes in another.) We then apply a trick of Benedetto for bounding the number of rational preperiodic points of a polynomial. In essence, his technique focuses on a place $p$ of genuinely bad reduction for $f_c$ and captures the filled Julia set inside a collection of at most~$2$ $p$-adic disks of moderate radius, or $4$~disks of very small radius. By the product formula, each such disk will contain at most one rational preperiodic point. This argument is highly reminiscent of the line of attack initiated by Call and Goldstine \cite{Call-Goldstine_1997} (and then refined and extended by Benedetto). The bound on the number of rational preperiodic points so produced is sharp enough to obtain the desired result for the graphs with 4 or 8 vertices. For the graphs $G$ with 6 vertices, we must also invoke non-existence results of Poonen \cite{Poonen_Preperiodic_Classification_1998} to exclude the occurrence of graphs that properly contain $G$. 

Benedetto's Trick for bounding the number of rational preperiodic points of a polynomial is explained in \S2. The proof of Theorem~\ref{Thm: Main} is given in \S3. We argue in \S4 that the theorem is best possible in a precise sense.

\bigskip

\noindent \textbf{Acknowledgments.} The author would like to thank Rob Benedetto for helpful conversations on the history of these techniques, and Michael Stoll and Jordan Ellenberg for the pointer to van der Corput's theorem (via MathOverflow question 123605).

%%%%%%%%%%%%%%%%%%%%%%%%%
%%%%%%%%%%%%%%%%%%%%%%%%%
	
\section{Benedetto's Trick}	

	The goal of this section is to explain (and extend) a technique of Benedetto for obtaining global bounds on $\#\PrePer(f, \QQ)$ using local data and the product formula. 

	Let $p \leq \infty$ be a prime of $\QQ$. We write $\QQ_p$ for the completion of $\QQ$ at $p$ and $|\cdot|_p$ for its standard absolute value; we write $\CC_p$ for the completion of an algebraic closure of $\QQ_p$. In particular, $\QQ_\infty = \RR$, and $\CC_\infty = \CC$. For a polynomial $f$ with rational coefficients, we define the  \textit{filled Julia set} for $f$ at $p$ to be
		\[
			\mathcal{K}_{f,p} = \{x \in \CC_p \ : \ f^n(x) \not\rightarrow \infty \text{ as $n \to \infty$}\}.
		\]
The key observation is that preperiodic points for $f$ lie in the filled Julia set $\mathcal{K}_{f,p}$ for every prime~$p$.

The following lemma provides stronger bounds for the filled Julia set in the case of finite $p$. It is an extension of \cite[Lem.~4.3]{Benedetto_Function_Fields_2005}; the proof strategy is essentially the same. 

\begin{lem}
\label{Lem: Small disks}
	Let $p$ be a finite prime. Let $f \in \QQ[z]$ be a monic polynomial of degree $ d \geq 2$, and let $r$ be the radius of the minimal closed disk containing the filled Julia set $\mathcal{K}_{f,p}$. For any point $x_0 \in \mathcal{K}_{f,p}$, the following containments hold:
	\[
		\mathcal{K}_{f,p} \subset \bigcup_{f(x) = x_0} D(x, 1) \qquad \text{and} \qquad
		\mathcal{K}_{f,p} \subset \bigcup_{f^2(x) = x_0} D(x, r^{- 1/d}).
	\]
\end{lem}

\begin{remark}
In \cite[Lem.~4.3]{Benedetto_Function_Fields_2005}, Benedetto proves that the filled Julia set $\mathcal{K}_{f,p}$ is contained in a union of at most $d^2$ open disks $D(x, 1)^- = \{y \in \CC_p : |y - x|_p < 1\}$ as $x$ varies over the solutions to $f^2(z) = x_0$.
\end{remark}

%\begin{remark}
%	Although we will not need it for our application, one can extend Benedetto's trick every further as follows. Fix $\ell \geq 0$ and set 
%	$\sigma = \frac{1 - d^{1-n}}{d-1}  \geq 0$. For any point $x_0 \in \mathcal{K}_f$, we have
%	\[
%		\mathcal{K}_f \subset \bigcup_{f^n(x) = x_0} D(x, r^{- \sigma}).
%	\]
%\end{remark}

\begin{proof}
	Write $U_0 \subset \CC_p$ for the minimal closed disk containing the filled Julia set of $f$; it has radius $r$ by assumption. Then $r \geq 1$, with equality if and only if $f$ has potential good reduction \cite[Lem.~2.5]{Benedetto_Global_Fields_2007}. If $r = 1$, then $U_0$ is totally invariant (i.e., $f^{-1}(U_0) = f(U_0) = U_0$), and the result is trivial. Let us assume in what remains that $r > 1$. 
	
	Write $x_1, \ldots, x_d$ for the elements of $f^{-1}(x_0)$, repeated according to multiplicity. For each $i = 1, \ldots, d$, denote by $x_{i,1}, \ldots, x_{i,d}$ the elements of $f^{-1}(x_i)$, again with multiplicity. % Then $f^{-2}(x_0) = \{x_{i,j} : 1 \leq i,j \leq d\}$.  
	
	Let $V_1, \ldots, V_m$ be the distinct pre-images of $U_0$. Since $f$ has genuinely bad reduction, we know $m > 1$ and each $V_i$ maps surjectively onto $U_0$. There must be two distinct disks $V_s$ and $V_t$ at distance exactly $r$, else the entire filled Julia set would fit into a disk of radius strictly less than $r$. Also, for each $i = 1, \ldots, d$, each of $V_s$ and $V_t$ must contain at least one of $x_{i,1}, \ldots, x_{i,d}$ since each of these disks maps surjectively onto $U_0$. Finally, for every $i = 1, \ldots, d$ and $w \in \CC_p$, there is some $j = 1, \ldots, d$ such that $|w - x_{i,j}|_p \geq r$; otherwise, all of the $x_{i,j}$ would be at distance less than $r$ by the ultrametric inequality, contradicting the fact that $V_s$ and $V_t$ are at distance $r$. 

	For the first containment in the statement, let $w \in \CC_p \smallsetminus \bigcup_{f(x) = x_0} D(x, 1)$. For any $i = 1, \ldots, d$, we see that  
		\[
			|f(w) - x_i|_p = \prod_{1 \leq j \leq d} |w - x_{i,j}|_p > 1^{d-1} \cdot r = r.
		\]
 This means $f(w) \not\in U_0$. By the total invariance of the filled Julia set, we conclude that $w \not\in \mathcal{K}_{f,p}$.

	A similar argument will prove the second containment in the statement. Let 
	\[
		w \in \CC_p \smallsetminus \bigcup_{f^2(x) = x_0} D(x, r^{-1/d}). 
	\]
For any $i = 1, \ldots, d$, we see that  
		\[
			|f(w) - x_i|_p = \prod_{1 \leq j \leq d} |w - x_{i,j}|_p > \left(r^{-1/d}\right)^{d-1} \cdot r = r^{1/d}.
		\]
Therefore
	\[
		\left|f^2(w) - x_0 \right|_p = \prod_{1 \leq i \leq d} \left| f(w) - x_i \right|_p >  r. 
	\]
As before, we conclude that $w \not\in \mathcal{K}_{f,p}$. 
\end{proof}

For the statement, write $M_{\QQ}$ for the set of primes of $\QQ$, including $p = \infty$. 
	
\begin{benTrick*}
	Let $d \geq 2$ be an integer, and let $(g_n)_{n \geq 1}$ be a sequence of monic polynomials of degree~$d$ with $\QQ$-coefficients.  
Suppose that for each $n$ there exists a partition 
	\[
		M_{\QQ} = A_n \cup B_n \cup C_n
	\]
for which the following hypotheses hold:
	\begin{itemize}
		\item \textup{(}Good Reduction\textup{)} For each prime $\ell \in A_n$, the minimum radius of a closed disk that contains $\mathcal{K}_{g_n,\ell}$ is~1. 
		\item \textup{(}Controlled Bad Reduction\textup{)} The set $B_n$ is stable for $n \geq 1$ and contains $\infty$. For each $\ell \in B_n$,  there exists $R_\ell \geq 1$
			independent of $n$ such that the $\ell$-adic filled Julia set $\mathcal{K}_{g_n, \ell} \cap \QQ_\ell$ is contained inside the disk
			  $D(0, R_\ell)$. 
		\item \textup{(}Uncontrolled Bad Reduction\textup{)} The cardinality of $C_n$ is stable. 
			For each $\ell \in C_n$, if $r_{n, \ell}$ is the minimum radius of a closed disk containing the filled 
			Julia set $\mathcal{K}_{g_n, \ell}$, then $\min \{r_{n,\ell} : \ell \in C_n\}$ tends to infinity with $n$.
	\end{itemize}	
If $\#C_n \geq 1$, then $\# \PrePer(g_n, \QQ) \leq d^{\#C_n + 1}$ for all $n$ sufficiently large. 
\end{benTrick*}

\begin{remark}
	Let $s$ be the number of odd primes of bad reduction. For our application in the next section, we have $d = 2$ and $\#C_n = s$, so that our bound will be $2^{s+1}$. The proof of \cite[Thm.~6.9]{Call-Goldstine_1997} produces the bound $2^{s+1}\cdot R_\infty$, which is  generally weaker. 
\end{remark}
	
\begin{proof}
	A preperiodic point for $g_n$ must lie in the filled Julia set $\mathcal{K}_{g_n, \ell}$ for each prime $\ell$ of $\QQ$. If $\alpha$ and $\beta$ are distinct rational preperiodic points for $g_n$ and $\ell \in A_n \cup B_n$, this means 
	\begin{equation}
	\label{Eq: An and Bn}
		|\alpha - \beta|_\ell \leq 
			\begin{cases}
				1 & \text{if }\ell \in A_n \\
				2^{\epsilon(\ell)}R_\ell & \text{if }\ell \in B_n,
			\end{cases}
	\end{equation} 
where $\epsilon(\ell) = 1$ if $\ell = \infty$ and $0$ otherwise. We now use Lemma~\ref{Lem: Small disks} to provide a bound for the primes in $C_n$.  

	Assign an ordering to the elements of $C_n$, say $C_n = \{\ell_0, \ldots, \ell_t\}$. Write $r_{n,\ell_0}$ for the minimum radius of a disk that contains the filled Julia set $\mathcal{K}_{g_n, \ell_0}$. We apply Lemma~\ref{Lem: Small disks} to obtain $d^2$~disks $D_{0,1}, \ldots, D_{0,d^2}$ of radius $r_{n, \ell_0}^{-1/d}$ whose union contains the filled Julia set. (These disks  depend implicitly on the index $n$.) For each prime $\ell_j$ with $j = 1, \ldots, t$, we apply Lemma~\ref{Lem: Small disks} to obtain $d$~disks $D_{j,1}, \ldots, D_{j, d}$ of radius~1 whose union contains the filled Julia set $\mathcal{K}_{g_n, \ell_j}$. 
	
	To each element $x \in \PrePer(g_n, \QQ)$, we may assign an address 
		\[
			(a_0, \ldots, a_t) \in \{1,\ldots, d^2\} \times \{1,\ldots, d\}^t
		\] 
so that $x \in D_{j, a_j}$ for each $j=0, \ldots, t$. We claim that for $n$ sufficiently large, no two rational preperiodic points share the same address. Evidently, this implies the desired inequality: 
	\[
		\# \PrePer(g_n, \QQ) \leq d^2 \cdot d^{t} = d^{\#C_n + 1},
	\]
To prove the claim, suppose that $\alpha$ and $\beta$ are distinct rational preperiodic points for $g_n$ with the same address. Then
	\begin{equation}
	\label{Eq: Cn}
		|\alpha - \beta|_{\ell_j} \leq 
			\begin{cases}				
				r_{n, \ell_0}^{-1/d} & \text{if }j = 0 \\			
				1 & \text{if }j = 1, \ldots, t. 
			\end{cases}
	\end{equation}
Combining \eqref{Eq: An and Bn} with \eqref{Eq: Cn} and applying the product formula gives
	\[
		1 = \prod_{\ell \leq \infty} | \alpha - \beta|_\ell \leq r_{n, \ell_0}^{-1/d}\prod_{\ell \in B_n} 2^{\epsilon(\ell)} R_\ell. 
	\]
Since $r_{n, \ell_0} \to \infty$ with $n$, we obtain a contradiction. 
\end{proof}

%%%%%%%%%%%%%%%%%%%%%%%%%
%%%%%%%%%%%%%%%%%%%%%%%%%

\section{Proof of the theorem}

	In order to apply Benedetto's Trick, we make use of the following well known result that gives bounds for the various filled Julia sets for the polynomial $f_c(z) = z^2 + c$. For example, see \cite[\S2]{JXD_2012} for proofs. For notation, if $p \leq \infty$ is a prime of $\QQ$, $a \in \CC_p$, and $s > 0$, we write $D(a,s) = \{x \in \CC_p : |x-a|_p \leq s\}$ for the closed disk of radius $s$. 
	
\begin{prop}
\label{Prop: Real/padic}
	Let that $c \in \QQ$.  
	\begin{itemize}
		\item If $c \leq 1/4$, then		
			$\mathcal{K}_{f_c, \infty} \cap \RR \subset D\left(0, \frac{1}{2} + \sqrt{\frac{1}{4} - c}\right) \cap \RR$.
		\item If $p < \infty$ is prime and $|c|_p \leq 1$, then $\mathcal{K}_{f_c, p} = D(0,1)$.
		\item If $p < \infty$ is prime and $|c|_p > 1$, then $\mathcal{K}_{f_c, p} \subset \{x \in \CC_p \ : \ |x|_p = |c|_p^{1/2} \}$. 
	\end{itemize} 
\end{prop}

	Each of the six graphs in the theorem requires a slightly different argument, so we will group them accordingly.

\subsection{The graph $4(1,1)$}
	
	For each odd rational prime $p$, we define
		\[
			c_p = \frac{1}{4} - \frac{1}{p^2}. 
		\]
One checks easily that the four points $\pm \frac{1}{2} \pm \frac{1}{p} $ are preperiodic for $f = f_{c_p}$, and that they yield the graph $4(1,1)$. In particular, 
	\[
		\#\PrePer(f, \QQ) \geq 4.
	\] 
We show the opposing inequality using Benedetto's Trick. 

	We use the partition of $M_\QQ$ determined by 
		\[
			A_p = M_\QQ \smallsetminus \{ 2, p, \infty\} \qquad B_p = \{2, \infty\} \qquad C_p = \{p\}. 
		\]
For $\ell \neq 2, p, \infty$,  the minimum radius of a closed disk containing the filled Julia set $\mathcal{K}_{f,\ell}$ is~1 (Proposition~\ref{Prop: Real/padic}). For $\ell = \infty$,  Proposition~\ref{Prop: Real/padic} shows that the real filled Julia set  $\mathcal{K}_{f,\infty} \cap \RR$ is contained in the closed disk about the origin with radius 
	\[
		\frac{1}{2} + \sqrt{\frac{1}{4} - c_p} = \frac{1}{2} + \frac{1}{p}  \leq \frac{5}{6}. 
	\]
Set $R_\infty = 5/6$. If $\ell = 2$, then $|c_p|_2 = 4$, so that 
	\[
		\mathcal{K}_{f,2} \subset \{x \in \CC_2 \ : \ |x|_2 = 2 \} \subset D(0,2).
	\]
Set $R_2 = 2$. Finally, if $\ell = p$, then $|c_p|_p = p^2$, so that
	\[
		\mathcal{K}_{f,p} \subset \{x \in \CC_p \ : \ |x|_p = p \}.
	\]
In particular, since $\frac{1}{2} \pm \frac{1}{p}$ are both preperiodic for $f$, and since they lie at distance $p$ from each other, the smallest disk that contains $\mathcal{K}_{f,p}$ has radius $r_p = p$. Evidently the hypotheses of Benedetto's Trick are met, so that for $p$ sufficiently large we obtain
	\[
		\#\PrePer(f, \QQ) \leq 2^{\#C_p + 1} = 4. 
	\]

%%%%%%%%%%%%%%%%%%%%%%%%%
%%%%%%%%%%%%%%%%%%%%%%%%%

\subsection{The graph $4(2)$}

	For each odd prime $p$, set 
	\[
		c_p = -\frac{3}{4} - \frac{1}{p^2}. 
	\]
The four points $\pm \frac{1}{2} \pm \frac{1}{p}$ are preperiodic for $f = f_{c_p}$, and they yield the graph $4(2)$. The rest of the proof is identical to the one used for the graph $4(1,1)$, with the exception that we take $R_\infty = \frac{1}{2} +\frac{1}{3}\sqrt{10}$. 

%%%%%%%%%%%%%%%%%%%%%%%%%
%%%%%%%%%%%%%%%%%%%%%%%%%

\subsection{The graph $6(1,1)$}

	The graph $\PrePer(f_c, \QQ)$ contains $6(1,1)$ if and only if 
		\[
			c = - \frac{2(\eta^2 + 1)}{(\eta^2 - 1)^2}
		\]
for some $\eta \in \QQ \smallsetminus \{0, \pm 1\}$ \cite[Thm.~3.2]{Poonen_Preperiodic_Classification_1998}. The known rational preperiodic points are
	\[
		\pm \frac{2\eta}{\eta^2 - 1}, \qquad \pm \frac{1}{2} \pm \frac{\eta^2 + 3}{2(\eta^2 - 1)}.
	\]
Make the change of variable $\eta = \frac{t+1}{t-1}$ to arrive at
	\[
		c = - \frac{t^4 - 2t^3 + 2t^2 - 2t + 1}{4t^2}. 
	\]
	
	Define two sequences $(p_n)$ and $(q_n)$ of odd primes such that $p_n \neq q_n$ and $p_n / q_n \to 1$ as $n \to \infty$. (This is possible, for example, by the Prime Number Theorem, since the $n$th prime is of approximate size $n \log n$.) We define $c_n$ by setting $t = p_n / q_n$ in the preceding formula:
	\[
		c_n = - \frac{p_n^4 - 2p_n^3q_n + 2p_n^2q_n^2 - 2p_nq_n^3 + q_n^4}{(2p_nq_n)^2}. 
	\]
For each such value $c_n$, the graph of preperiodic points for $f = f_{c_n}$ contains $6(1,1)$ by construction. Indeed, the six known preperiodic points are $\pm \frac{1}{2} \pm \frac{p_n^2 -p_nq_n + q_n^2}{2p_nq_n}$ and $\pm \frac{p_n^2 - q_n^2}{2p_nq_n}$.   

	Consider the partition of $M_\QQ$ given by
		\[
			A_n = M_\QQ \smallsetminus \{2, p_n, q_n, \infty\} \qquad B_n = \{2, \infty\} \qquad C_n = \{p_n, q_n\}. 
		\]
For $\ell \in A_n$, we see that $|c_n|_\ell \leq 1$, so that $\mathcal{K}_{f, \ell} = D(0,1)$ as desired. For $\ell = \infty$, the real filled Julia set $\mathcal{K}_{f,\infty} \cap \RR$ is contained in the disk about the origin of radius
		\[
			\frac{1}{2} + \sqrt{\frac{1}{4} - c}  = \frac{1}{2} + \frac{p_n^2-p_nq_n + q_n^2}{2p_nq_n} \to 1. 
		\]
Hence, there is $R_\infty > 1$ so that $\mathcal{K}_{f, \infty} \cap \RR \subset D(0, R_\infty)$ for all $n$.  For $\ell = 2$, we may take $R_2 = 2$ (Proposition~\ref{Prop: Real/padic}). Finally, for $\ell \mid p_nq_n$, we have $|c_n|_\ell = \ell^2$. There are distinct preperiodic points at distance $\ell$ from each other, so it follows that the smallest closed disk containing $\mathcal{K}_{f, \ell}$ has radius $r_\ell = \ell$. Applying Benedetto's Trick shows that for $n$ sufficiently large we have the bound
	\[
		\#\PrePer(f_{c_n}, \QQ) \leq 2^{\#C_n + 1} = 8. 
	\]
	
	To complete the proof, let us assume that $n$ is sufficiently large that $6 \leq \#\PrePer(f_{c_n}, \QQ) \leq 8$. If the lower bound is strict, then there exists a point $x \in \Aff^1(\QQ)$ that is preperiodic for $f = f_{c_n}$ and that is not represented in the graph $6(1,1)$. If the orbit of $x$ contains an $m$-cycle $\{x_1, \ldots, x_m\}$ for some $m \geq 2$, then the points $\{-x_1, \ldots, -x_m\}$ are also preperiodic and map into this $m$-cycle (because $f(z) = f(-z)$). Hence, $\#\PrePer(f, \QQ) \geq 6 + 2m > 8$, a contradiction. Therefore, the orbit of $x$ must contain a fixed point. As $f$ has at most 2 fixed points, it follows that $f(\pm x)$ is an element of the graph $6(1,1)$. 
	
	Let us say that the point $x$ is of type $a_b$ if $f^b(x)$ lies in an $a$-cycle, but $f^{b-1}(x)$ does not. By the preceding paragraph, the point $x$ is of type $1_2$ or $1_3$. The latter case is immediately ruled out by \cite[Thm.~3.6]{Poonen_Preperiodic_Classification_1998}. In the former case, $f$ has four points of type $1_2$, which is ruled out by  \cite[Thm.~3.2]{Poonen_Preperiodic_Classification_1998}. We conclude that $\#\PrePer(f_{c_n}, \QQ) = 6$, as desired.

%%%%%%%%%%%%%%%%%%%%%%%%%
%%%%%%%%%%%%%%%%%%%%%%%%%

\subsection{The graph $6(2)$}

	The strategy in this case is virtually identical to the one used in the preceding section, so we will simply sketch the differences. 
	
	The graph $\PrePer(f_c, \QQ)$ contains $6(2)$ if and only if 
		\[
			c = - \frac{\nu^4 + 2\nu^3 + 2\nu^2 - 2\nu + 1}{(\nu^2 - 1)^2}
		\]
for some $\nu \in \QQ \smallsetminus \{0, \pm 1\}$ \cite[Thm.~3.3]{Poonen_Preperiodic_Classification_1998}. Making the change of coordinate $\nu= \frac{t + 1}{t - 1}$ yields
	\[
		c = - \frac{t^4 + 2t^3 + 2t^2 - 2t + 1}{4t^2}.
	\]
Choose two sequences of primes $(p_n)$ and $(q_n)$ as before and set $t = p_n / q_n$. Then 
	\[
		c_n = - \frac{p_n^4 + 2p_n^3 q_n + 2p_n^2 q_n^2 - 2p_nq_n^3 + q_n^4}{(2p_n q_n)^2},
	\]
and the six preperiodic points that arise in this construction are $\pm \frac{p_n^2 + q_n^2}{2p_nq_n}$ and $\pm \frac{1}{2} \pm \frac{p_n^2 + p_nq_n - q_n^2}{2p_nq_n}$. 

	At the infinite prime, we see that $c_n \to -1$ with $n$. Thus there is $R_\infty > \frac{1}{2}\left(1 + \sqrt{5}\right)$ so that the real filled Julia set $\mathcal{K}_{f_{c_n}, \infty} \cap \RR$ is contained in the closed disk about the origin of radius $R_\infty$. The remainder of the application of Benedetto's Trick is identical to the preceding section. We conclude that $6 \leq \#\PrePer(f_{c_n}, \QQ) \leq 8$ for all sufficiently large $n$. If the lower inequality is sharp, then there is a rational preperiodic point $x$ of type $2_2$ or $2_3$ that has not been accounted for. Both of these cases are outlawed by \cite[Thm.~3.3, 3.6]{Poonen_Preperiodic_Classification_1998}. Hence, $\#\PrePer(f_{c_n}, \QQ) = 6$, as desired.

%%%%%%%%%%%%%%%%%%%%%%%%%
%%%%%%%%%%%%%%%%%%%%%%%%%

\subsection{The graph $6(3)$}

	According to \cite[Thm.~1.3]{Poonen_Preperiodic_Classification_1998}, the graph $\PrePer(f_c, \QQ)$ contains $6(3)$ if and only if 
		\begin{equation}
		\label{Eq: three poles}
			c = -\frac{ \tau^6 + 2\tau^5 + 4\tau^4 + 8 \tau^3 + 9 \tau^2 + 4\tau + 1}{4\tau^2(\tau+1)^2}
		\end{equation}
for some $\tau \in \QQ \smallsetminus \{0,-1\}$.  The six preperiodic points that one obtains by taking such a $c$ are 
	\[
		\pm \frac{\tau^3 + 2\tau^2 + \tau + 1}{2\tau (\tau+1)}, \qquad \pm \frac{\tau^3 - \tau - 1}{2\tau (\tau+1)},
		\qquad 		\pm \frac{\tau^3 + 2\tau^2 + 3\tau + 1}{2\tau (\tau+1)}. 
	\]

	As $c$ has three poles as a rational function in $\tau$, controlling the primes that occur in its denominator is a trickier endeavor that in the previous sections. We make use of the following result from analytic number theory:

\begin{thm}[van der Corput, 1939]
	For all $N$ sufficiently large, there exists a 3-term arithmetic progression $p, p+k, p+2k$ of primes in the interval $(N, 2N)$. 
\end{thm}
	
\begin{proof}
	The proof that there are infinitely many 3-term arithmetic progressions of primes is given in \cite{VanDerCorput_3APs}. In a private communication, Andrew Granville has informed the author that one may impose the added condition $N < p, p+k, p+2k < 2N$ in the proof. 
\end{proof}

	The theorem implies that there exist 3-term arithmetic progressions $p, p+k, p+2k$ with $p$ arbitrarily large and $k < p$. For such an arithmetic progression, it follows that 
	\[
		\frac{1}{3} < \frac{p}{p+2k} < 1.
	\]
Let $(p_n)$ and $(q_n)$ be sequences of primes such that $p_n \to \infty$ and $p_n < \frac{p_n + q_n}{2} < q_n$ is a 3-term arithmetic progression as above. Define $c_n \in \QQ$ by setting $\tau = p_n / q_n$ in \eqref{Eq: three poles}: 
	\[
		c_n = - \frac{p_n^6 + 2p_n^5q_n + 4p_n^4q_n^2 + 8 p_n^3q_n^3 + 9 p_n^2q_n^4 + 4p_nq_n^5 + q_n^6}{\left[2p_nq_n(p_n + q_n)\right]^2}. 
	\]	
	
	Partition $M_\QQ$ using the sets
		\[
			A_n = M_\QQ \smallsetminus \{2, p_n, (p_n + q_n)/2, q_n, \infty\} \qquad B_n = \{2, \infty\} \qquad C_n = \{p_n, (p_n + q_n)/2, q_n\}.
		\]
	Set $f = f_{c_n}$. Since the ratio $p_n / q_n$ is positive and bounded away from zero and infinity, it follows that $|c_n|_\infty$ is bounded in the Archimedean metric, and hence there is $R_\infty > 1$ such that $\mathcal{K}_{f,\infty} \cap \RR$ is contained in the disk $D(0, R_\infty)$ for all $n$ (Proposition~\ref{Prop: Real/padic}).  The remainder of the application of Benedetto's Trick is similar to the section on the graph $6(1,1)$. We conclude that 
	\[
		6 \leq \#\PrePer(f_{c_n}, \QQ) \leq 2^{\#C_n + 1} = 16
	\] 
for all sufficiently large $n$. If the lower inequality is sharp, we let $x$ be a rational preperiodic point that does not lie in the graph $6(3)$ (of known points). Then the orbit of $x$ contains an $m$-cycle $\{x_1, \ldots, x_m\}$ for some $m \geq 1$ as well as the points $\{-x_1, \ldots, -x_m\}$. Now $m < 6$, since otherwise we have $ \#\PrePer(f_{c_n}, \QQ) \geq 2m + 6 > 16$. By results of Morton \cite{Morton_Arithmetic_Properties_II_1998} and Flynn/Poonen/Schaefer \cite{Poonen-Flynn-Schaefer_1997},  there does not exist a rational point of period $4$ or $5$ for any parameter $c \in \QQ$. Hence $m < 4$. The (unconditional) results in \cite{Poonen_Preperiodic_Classification_1998} now show that $f_{c_n}$ has at most one periodic orbit of period~3, and that it has no other periodic orbit. That is, $\#\PrePer(f_{c_n}, \QQ) = 6$ for $n$ sufficiently large.

%%%%%%%%%%%%%%%%%%%%%%%%%
%%%%%%%%%%%%%%%%%%%%%%%%%

\subsection{The graph $8(2,1,1)$}

	According to \cite[Thm.~2.1]{Poonen_Preperiodic_Classification_1998}, the graph $\PrePer(f_c, \QQ)$ contains $8(2,1,1)$ if and only if 
		\[
			c = - \frac{3\mu^4 + 10 \mu^2 + 3}{4(\mu^2 - 1)^2}
		\]
for some $\mu \in \QQ \smallsetminus \{0, \pm 1\}$. The eight preperiodic points that one obtains by taking such a $c$ are 
	\[
		\pm \frac{1}{2} \pm \frac{\mu^2 + 1}{\mu^2 - 1}, \qquad \pm\frac{1}{2} \pm \frac{2\mu}{\mu^2 - 1}.
	\]	
Make the change of variable $\mu = \frac{1+t}{1-t}$ to get
	\[
		c = - \frac{t^4 + t^2 + 1}{4t^2}.
	\]

	Define two sequences $(p_n)$ and $(q_n)$ of odd primes such that $p_n \neq q_n$ and $p_n / q_n \to 1$ as $n \to \infty$. We define $c_n$ by setting $t = p_n / q_n$ in the preceding formula:
	\[
		c_n = - \frac{p_n^4 + p_n^2 q_n^2 + q_n^4}{4p_n^2 q_n^2}. 
	\]
For each such value $c_n$, the graph of preperiodic points for $f = f_{c_n}$ contains $8(2,1,1)$ by construction. Indeed, the eight known preperiodic points are $\pm \frac{1}{2} \pm \frac{p_n^2 \pm q_n^2}{2p_nq_n}$.   

	Consider the partition of $M_\QQ$ given by
		\[
			A_n = M_\QQ \smallsetminus \{2, p_n, q_n, \infty\} \qquad B_n = \{2, \infty\} \qquad C_n = \{p_n, q_n\}. 
		\]
For $\ell \in A_n$, we see that $|c_n|_\ell \leq 1$, so that $\mathcal{K}_{f, \ell} = D(0,1)$ as desired. For $\ell = \infty$, the real filled Julia set $\mathcal{K}_{f,\infty} \cap \RR$ is contained in the disk about the origin of radius
		\[
			\frac{1}{2} + \sqrt{\frac{1}{4} - c}  = \frac{1}{2} + \frac{p_n^2+q_n^2}{2p_nq_n} \to \frac{3}{2}. 
		\]
Hence, there is $R_\infty > 3/2$ so that $\mathcal{K}_{f, \infty} \cap \RR \subset D(0, R_\infty)$ for all $n$.  For $\ell = 2$, we may take $R_2 = 2$ (Proposition~\ref{Prop: Real/padic}). Finally, for $\ell \mid p_nq_n$, we have $|c_n|_\ell = \ell^2$. There are distinct preperiodic points at distance $\ell$ from each other, so it follows that the smallest closed disk containing $\mathcal{K}_{f, \ell}$ has radius $r_\ell = \ell$. Applying Benedetto's Trick shows that for $n$ sufficiently large we have the bound
	\[
		\#\PrePer(f_{c_n}, \QQ) \leq 2^{\#C_n + 1} = 8. 
	\]
	
%%%%%%%%%%%%%%%
%%%%%%%%%%%%%%%

\section{Optimality}

	Recall from the introduction that a finite directed graph $G$ is \textbf{admissible} if there is a parameter $c \in \bar \QQ$ such that $\PrePer(f_c, \bar \QQ) \supset G$. Theorem~\ref{Thm: Main} is best possible in the following sense.
	
\begin{prop}
	Let $G$ be an admissible graph that is isomorphic to $\PrePer(f_c, \QQ)$ for infinitely many parameters $c \in \QQ$. Then $G$ is either empty, or it is among the six listed in Theorem~\ref{Thm: Main}.
\end{prop}

\begin{proof}
	The empty graph occurs for infinitely many parameters $c \in \QQ$. For example, set $c = 1/p$ for $p$ a prime and observe that any preperiodic point must be of the form $x = u/\sqrt{p}$ for some $u \in \ZZ \smallsetminus \{0\}$ (Proposition~\ref{Prop: Real/padic}). Hence $\PrePer(f_c, \QQ) = \varnothing$. In what remains, we may assume that $G$ is nonempty. 

The graph $G$ may not contain a directed cycle of length at least~4. For otherwise, Morton's genus formula \cite[Thm.~C]{Morton_Algebraic_Curves_1996} shows that the corresponding algebraic curve $X_G$ has genus at least~2, and hence only finitely many rational points. If $f_c$ admits no rational periodic point of period greater than 3, then Poonen has shown that the graph $\PrePer(f_c, \QQ)$ is one of 12 types \cite[Thm.~3, Fig.~1]{Poonen_Preperiodic_Classification_1998}. One of his graphs is empty; six correspond to the ones appearing in Theorem~\ref{Thm: Main}; and he proves that the remaining five types may occur for only finitely many rational parameters $c \in \QQ$. 
\end{proof}

%%%%%%%%%%%%%%%

\bibliographystyle{plain}
\bibliography{xander_bib}

\providecommand\biburl[1]{\texttt{#1}}\def\cprime{$'$}
\begin{thebibliography}{10}

\bibitem{Benedetto_Function_Fields_2005}
Robert~L. Benedetto.
\newblock Heights and preperiodic points of polynomials over function fields.
\newblock {\em Int. Math. Res. Not.}, (62):3855--3866, 2005.

\bibitem{Benedetto_Global_Fields_2007}
Robert~L. Benedetto.
\newblock Preperiodic points of polynomials over global fields.
\newblock {\em J. Reine Angew. Math.}, 608:123--153, 2007.

\bibitem{Call-Goldstine_1997}
Gregory~S. Call and Susan~W. Goldstine.
\newblock Canonical heights on projective space.
\newblock {\em J. Number Theory}, 63(2):211--243, 1997.

\bibitem{JXD_2012}
John~R. Doyle, Xander Faber, and David Krumm.
\newblock Computation of preperiodic structures for quadratic polynomials over
  number fields.
\newblock Preprint, 2012.

\bibitem{JXD_2013_I}
John~R. Doyle, Xander Faber, and David Krumm.
\newblock Analysis of preperiodic structures for quadratic polynomials over
  quadratic fields, {I}.
\newblock In preparation, 2013.

\bibitem{Poonen-Flynn-Schaefer_1997}
E.~V. Flynn, Bjorn Poonen, and Edward~F. Schaefer.
\newblock Cycles of quadratic polynomials and rational points on a genus-{$2$}
  curve.
\newblock {\em Duke Math. J.}, 90(3):435--463, 1997.

\bibitem{Hutz_Ingram_2013}
Benjamin Hutz and Patrick Ingram.
\newblock Numerical evidence for a conjecture of {P}oonen.
\newblock Preprint, arXiv:0909.5050 [math.DS], to appear in \textit{{R}ocky
  {M}ountain {J}. {M}ath.}, 2013.

\bibitem{Morton_Algebraic_Curves_1996}
Patrick Morton.
\newblock On certain algebraic curves related to polynomial maps.
\newblock {\em Compositio Math.}, 103(3):319--350, 1996.

\bibitem{Morton_Arithmetic_Properties_II_1998}
Patrick Morton.
\newblock Arithmetic properties of periodic points of quadratic maps. {II}.
\newblock {\em Acta Arith.}, 87(2):89--102, 1998.

\bibitem{Morton_Silverman_1994}
Patrick Morton and Joseph~H. Silverman.
\newblock Rational periodic points of rational functions.
\newblock {\em Internat. Math. Res. Notices}, (2):97--110, 1994.

\bibitem{Poonen_Preperiodic_Classification_1998}
Bjorn Poonen.
\newblock The classification of rational preperiodic points of quadratic
  polynomials over {${\bf Q}$}: a refined conjecture.
\newblock {\em Math. Z.}, 228(1):11--29, 1998.

\bibitem{Stoll_6-cycles_2008}
Michael Stoll.
\newblock Rational 6-cycles under iteration of quadratic polynomials.
\newblock {\em LMS J. Comput. Math.}, 11:367--380, 2008.

\bibitem{VanDerCorput_3APs}
J.~G. van~der Corput.
\newblock \"{U}ber {S}ummen von {P}rimzahlen und {P}rimzahlquadraten.
\newblock {\em Math. Ann.}, 116(1):1--50, 1939.

\bibitem{Walde_Russo}
Ralph Walde and Paula Russo.
\newblock Rational periodic points of the quadratic function {$Q_c(x)=x^2+c$}.
\newblock {\em Amer. Math. Monthly}, 101(4):318--331, 1994.

\end{thebibliography}

\end{document}